\documentclass[final,leqno,onefignum,onetabnum]{siamltex}
\usepackage{amsmath}
\usepackage{amssymb}
\usepackage{graphicx}
\usepackage[linesnumbered,ruled]{algorithm2e}
\usepackage{booktabs}
\usepackage{hyperref}
\usepackage{cite}

\newcommand{\n}[1]{\boldsymbol{#1}}
\newcommand{\cI}{\mathcal{I}}
\newcommand{\cA}{\mathcal{A}}
\newcommand{\cO}{\mathcal{O}}
\newcommand{\Idn}{\mathcal{I}_{\mathrm{d}}^{n}}

\newtheorem{example}{Example}[section]

\renewcommand{\theenumi}{\thesubsection.\arabic{enumi}}
\def\BState{\State\hskip-\ALG@thistlm}

\begin{document}
\title{Computing the B\'ezier Control Points of the Lagrangian Interpolant in Arbitrary Dimension}

\author{Mark Ainsworth \footnotemark[2] and Manuel A. S\'anchez \footnotemark[2]}

\maketitle
\begin{abstract}
The Bernstein-B\'ezier form of a polynomial is widely used in the fields of
computer aided geometric design, spline approximation theory and, more
recently, for high order finite element methods for the solution of partial
differential equations. However, if one wishes to compute the classical
Lagrange interpolant relative to the Bernstein basis, then the resulting
Bernstein-Vandermonde matrix is found to be highly ill-conditioned.

In the univariate case of degree $n$, Marco and Martinez~\cite{MR2305145}
showed that using Neville elimination to solve the system exploits the total
positivity of the Bernstein basis and results in an $\mathcal{O}(n^2)$
complexity algorithm. Remarkable as it may be, the Marco-Martinez algorithm has
some drawbacks: The derivation of the algorithm is quite technical; the
interplay between the ideas of total positivity and Neville elimination are not
part of the standard armoury of many non-specialists; and, the algorithm is
strongly associated to the univariate case.

The present work addresses these issues. An alternative algorithm for the
inversion of the univariate Bernstein-Vandermonde system is presented that has:
The same complexity as the Marco-Martinez algorithm and whose stability does
not seem to be in any way inferior; a simple derivation using only the basic
theory of Lagrange interpolation (at least in the univariate case); and, a
natural generalisation to the multivariate case.
\end{abstract}
\date{}

\renewcommand{\thefootnote}{\fnsymbol{footnote}}
\footnotetext[2]{Division of Applied Mathematics, Brown University, Providence,
RI 02912, USA (mark\_ainsworth@brown.edu, manuel\_sanchez\_uribe@brown.edu).
MA gratefully acknowledges the partial support of this work under AFOSR
contract FA9550-12-1-0399. \hfill}

\section{Introduction}

The classical Lagrangian interpolation problem consists of finding a polynomial
$p$ of degree at most $n$, such that
\begin{equation}\label{polyfit}
p(x_j) =  f_j, \quad j=0,...,n.
\end{equation}
where $\{x_j\}_{j=0}^n \subseteq [0,1]$ are distinct $(n+1)$ nodes and
$\{f_j\}_{j=0}^n$ are given data. The existence of a unique solution to this
problem is well-known from one's first course in numerical analysis. The
seasoned reader might also recall that, despite the uniqueness of the
interpolant itself, the choice of basis for constructing the Lagrange
interpolant is less clear-cut~\cite{MR2115059}--possibilities include the
Lagrangian or the Newton bases, with the much maligned monomial basis even
having a shout~\cite{MR1927606}.

The Bernstein basis functions for the space $\mathbb{P}^n([0,1])$ of polynomials of
degree at most $n$ on $[0,1]$, is given by
\begin{equation}
B_{k}^{n}(x) =  \binom{n}{k} (1-x)^{n-k} x^{k},\quad k=0,...,n,\quad x\in[0,1].
\end{equation}
The Bernstein functions extend naturally to give a basis for polynomials of
total degree at most $n$ on simplices in arbitrary numbers of spatial
dimensions (see Section \ref{section simplex}).

The basis has many unexpected and attractive properties that have led to it
being: almost ubiquitous in the computer aided geometric design (CAGD)
community~\cite{Farin:2001:CSC:501891,MR2921860} for the representation of
curves and surfaces; as an important theoretical tool in the spline
approximation theory literature~\cite{MR2355272}; and, more recently, a
practical tool for the efficient implementation of high order finite element
methods for the solution of partial differential
equations~\cite{MR3296172,MR3183046,MR2862005,MR2917163,MR2776913}. For further
information on these, and numerous other applications of Bernstein polynomials,
we refer the reader to the survey article of Farouki~\cite{MR2921860}. 

The Bernstein-B\'ezier (BB) representation of a polynomial
$p\in\mathbb{P}^n([0,1])$ takes the form $p=\sum_{k=0}^n c_k B_{k}^{n}$ in
which the coefficients $\{c_k\}$ are referred to as \emph{control
points}~\cite{Farin:2001:CSC:501891,MR2921860,MR2355272}, and which may be
associated with the $(n+1)$ uniformly spaced points on $[0,1]$ (even if the
interpolation nodes are non-uniformly spaced). However, while the polynomial
$p$ satisfies $p(0)=c_0$ and $p(1)=c_n$, this property does not hold at the
remaining control points. One one hand, this property does not hinder the
typical workflow of a CAGD engineer whereby the locations of $\{c_k\}$ are
adjusted until a curve of the desired shape is obtained. In effect, the control
points are used to \emph{define} the curve directly.  On the other hand, the
typical usage of polynomials in scientific computing is rather different in
that one generally wishes to \emph{fit} a polynomial to a given function.  For
example, in applying the finite element method to approximate the solution of a
partial differential equation, one might have boundary or initial data and
require to choose an appropriate (piecewise) polynomial approximation of the
data. The approximation is often chosen to be an interpolant, leading to what
we shall term the Bernstein-B\'ezier interpolation problem, which consists of
computing the control points $\{c_k\}_{k=0}^{n}$ such that the associated
Bernstein-B\'ezier polynomial interpolates the data:
\begin{equation}\label{Bernsteinrepresentation}
    p(x_j) = \sum_{k=0}^n c_k B_k^{n}(x_j) = f_j, \quad j=0,\ldots,n .
\end{equation}

Conditions~\eqref{Bernsteinrepresentation} may be expressed as a system of
linear equations involving the \emph{Bernstein-Vandermonde}
matrix~\cite{MR2305145}. If the monomial basis were to be used, then the
standard Vandermonde matrix would emerge. The highly ill-conditioned nature of
the Vandermonde matrix is well-documented~\cite{MR0380189}. Notwithstanding, the inversion of the Vandermonde matrix to compute the
Lagrangian interpolant is in some ways preferable to more direct
methods~\cite{MR892779,MR1927606}.

The Bernstein-Vandermonde matrix is likewise found to be highly
ill-conditioned~\cite{MR2305145} suggesting that its inversion may not be the
wisest approach. Nevertheless, \emph{structure} of the matrix arising from the
\emph{total positivity} of the Bernstein basis means that using Neville
elimination~\cite{MR1149743} to solve the system obviates some of the issues
due to ill-conditioning. Marco and Martinez~\cite{MR2305145} exploit this fact
and derive an algorithm for the inversion of the matrix that has
$\mathcal{O}(n^2)$ complexity---the same as multiplying by the inverse of the
matrix.

Remarkable though the Marco-Martinez algorithm may be, it does have its drawbacks:
\begin{itemize}
\item the derivation of the algorithm is highly technical involving non-trivial
	identities for the minors of Bernstein-Vandermonde matrices;
\item the interplay between the ideas of total positivity and Neville
	elimination are not part of the standard armoury of many
	non-specialists;
\item the algorithm seems to be firmly rooted to the solution of the Bernstein
	interpolation problem in the univariate case (indeed, total positivity
	is essentially a univariate concept).
\end{itemize}
The purpose of the present work is to address these issues. Specifically, we
shall present an alternative algorithm for the inversion of the univariate
Bernstein-Vandermonde system that has:
\begin{itemize}
\item the same complexity as the Marco-Martinez algorithm and whose stability
	does not seem to be in any way inferior;
\item a simple derivation that could be taught to undergraduates familiar with
	only the basic theory of Lagrange interpolation (at least in the
	univariate case);
\item a natural generalisation to the multivariate case (essential for
	applications such as finite element analysis in two and three
	dimensions).
\end{itemize}
The remainder of this article is organised as follows. The next section deals
with deriving the new algorithm in the univariate case using only elementary
facts about univariate interpolation and basic properties of Bernstein
polynomials.  Section 3 shows how the univariate algorithm is easily extended
to solve the multivariate Bernstein interpolation problem in the case of tensor
product polynomials.

In Section 4, we tackle the much harder task of solving the Bernstein
interpolation problem on simplices in higher dimensions in which even the
existence of the Lagrangian interpolant is less
obvious~\cite{MR2200737,MR892779}. Using the standard hypothesis under which
the Lagrange interpolant exists, we develop an algorithm for the solution of
the multivariate Bernstein interpolation problem on simplices in two
dimensions, and indicate how it may be extended to arbitrary dimensions. The
ideas used in Section 4 are, modulo technical issues, essentially the same to
those used in Section 2 to handle the univariate case. Sections 2-4 are accompanied by some numerical examples illustrating the behaviour and stability of the resulting
algorithms.

In summary, the algorithms developed in the present work form a key step
towards a more widespread use of Bernstein-B\'ezier techniques in scientific
computing in general, and in finite element analysis in particular by enabling
the use of non-homogeneous initial and boundary data. More generally, the
problem of how to extend the Marco-Martinez algorithm to the solution of
Bernstein-Vandermonde systems to arbitrary dimension is addressed.

\section{Bernstein polynomial interpolation in 1D}
\subsection{Analytical solution}
The \emph{improved Lagrange formula}~\cite{MR2115059} for the Lagrangian interpolant is given by
\begin{equation}\label{improvedLagrange}
p(x)=\sum_{j=0}^{n} \mu_j f_j \frac{\ell(x)}{x-x_{j}}, \quad x\in [0,1].
\end{equation}
where $\ell\in\mathbb{P}^{n+1}([0,1])$ and the barycentric weights $\mu_j$,
$j=0,...,n$ are given by
\begin{equation}
\ell(x)= \prod_{k=0}^n(x-x_k),\qquad\mu_{j} =   \frac{1}{\ell'(x_j)} \mathrm{ for
} j=0,...,n.
\end{equation}
Our first result gives a closed form for the control points satisfying
\eqref{Bernsteinrepresentation}:
\begin{theorem}\label{theoremLagrange}
For $k=0,...,n$ define the control points $\{c_k\}_{k=0}^{n}$ by the rule
\begin{equation}
c_{k}=\sum_{j=0}^{n} \mu_{j} f_{j} \tilde{w}_{k}(x_{j}),\quad k=0,...,n,
\end{equation}
where
\begin{equation}\label{w_k(z)}
  \tilde{w}_{k}(z)= \frac{w_{k}(z)}{z(1-z)B_{k}^{n}(z)}
  \mbox{ and } w_{k}(z)= -\sum_{i=0}^{k} a_{i} B_{i}^{n+1}(z)
\end{equation}
for $z$ any node $x_j$, with $\{a_i\}$ control points of $\ell$. Then, $\{c_k\}_{k=0}^{n}$ are the control points for
the Bernstein interpolant~\eqref{Bernsteinrepresentation}.
\end{theorem}
\begin{proof}
We claim that $\{\tilde{w}_{k}(x_j)\}_{k=0}^{n}$ are the control points of the
polynomial ${\ell(x)}/(x-x_j)$ for $j = 0,...,n$. Assuming the claim holds,
then
\begin{equation*}
p(x) = \sum_{k=0}^{n}c_k B_{k}^n(x)
     = \sum_{j=0}^{n} \mu_{j} f_{j} \sum_{k=0}^{n} w_{k}(x_{j})B_{k}^n(x)
     = \sum_{j=0}^{n} \mu_{j} f_{j} \frac{\ell(x)}{x-x_j},
\end{equation*}
and the improved Lagrange formula \eqref{improvedLagrange} then shows that $p$
is the interpolant.

It remains to prove the claim holds. Let $z$ be any node $x_j$,
then there exist constants $\tilde{w}_k$ such that
$
	\ell(x) = (x-z) \sum_{k=0}^{n}\tilde{w}_{k} B_{k}^{n}(x).
$
Now, using the definition of the Bernstein polynomials we obtain
$x-z = (1-z)B_{1}^{1}(x)-zB_{0}^{1}(x)$, and
\begin{equation}\label{eq:manuel}
	B_1^1 B_k^n = \frac{k+1}{n+1} B_{k+1}^{n+1}
	\mbox{ and }
	B_0^1 B_k^n = \left(1-\frac{k}{n+1}\right) B_{k}^{n+1}.
\end{equation}
Hence,
\begin{align*}
    \ell(x) =  \sum_{k=0}^{n+1}\left(
                  (1-z)\frac{k}{n+1}\tilde{w}_{k-1}
                -  z \Big(1-\frac{k}{n+1}\Big)\tilde{w}_{k}
		\right) B_{k}^{n+1}(x)
\end{align*}
where we define $\tilde{w}_{-1} = \tilde{w}_{n+1} = 0$. Hence, denoting by $\{a_i\}_{i=0}^{n+1}$ the B\'ezier control points of $\ell$
\begin{equation}\label{tildewsystem}
     a_{k} = (1-z)\frac{k}{n+1}\tilde{w}_{k-1}-z\Big(1-\frac{k}{n+1}\Big)\tilde{w}_{k},
     \quad k=0,...,n+1.
\end{equation}
Multiplying equation~\eqref{tildewsystem} by $B_{k}^{n+1}(z)$ and using the
definition of the Bernstein polynomials gives
\begin{align}\label{wtildestarsystem}
     a_{k} B_{k}^{n+1}(z)
     &= z(1-z)\Big( \tilde{w}_{k-1}B_{k-1}^{n}(z) - \tilde{w}_{k}B_{k}^{n}(z) \Big)  \\
     &= w_{k-1} - w_{k}, \quad k=0,...,n+1,
\end{align}
where $w_{k} = z(1-z)\tilde{w}_{k} B_{k}^{n}(z)$. The overdetermined system of
equations~\eqref{wtildestarsystem} is consistent since $\ell(z) =
\sum_{i=0}^{n+1} a_{i} B_{i}^{n+1}(z) = 0$.  Hence,	
\begin{equation}
     w_{k} = -\sum_{i=0}^{k} a_{i}B_{i}^{n+1}(z), \quad k=0,...,n
\end{equation}
and the result follows as claimed.
\end{proof}

Theorem \ref{theoremLagrange} gives a new explicit formula for the solution of
the Bernstein interpolation problem based on the Lagrangian form of the
interpolant. However, direct use of Theorem~\ref{theoremLagrange} for
computation of the actual solution would cost $\mathcal{O}(n^{3})$ operations.

\subsection{A simple algorithm for computing the univariate control points}\label{Newton1D}
Many elementary numerical analysis textbooks extol the virtues of using the
Newton formula \cite{MR0380189} for the polynomial interpolant
satisfying~\eqref{polyfit}:
\begin{equation}\label{Newtonsformula}
   p(x) =  \sum_{j=0}^n f[x_{0},...,x_{j}] w_j(x)
\end{equation}
where $w_{0}(x) =1$ and $w_{j}(x)= \prod_{k=0}^{j-1} (x-x_{k})$, for
$j=1,...,n$, and $f[x_{0},...,x_{j}]$ are the divided differences defined
recursively as follows:
\begin{equation*}
    f[x_{j},..., x_{k}] =  \frac{f[x_{j+1},..., x_{k}]  - f[x_{j},..., x_{k-1}] }{x_{k}-x_{j}},
    \quad k=j+1,...,n \mbox{ and }j = 0,...,n, 
\end{equation*}
whilst if $k = j$, then the value is simply  $f_j$.

Is there any advantage to using the Newton form of the interpolant for
Bernstein interpolation?  

\subsubsection{Newton-Bernstein Algorithm}
Let $p_k\in\mathbb{P}^{k}([0,1])$ be the Newton form of the interpolant at
the nodes $\{x_{j}\}_{j=0}^{k}$; that is
\begin{equation}\label{p_k}
   p_k(x)=\sum_{j=0}^k f[x_{0},...,x_{j}] w_{j}(x),\quad k=0,...,n.
\end{equation}

\begin{theorem}\label{theoremNB}
For $k=0,...,n$ define $\{w^{(k)}_{j}\}_{j=0}^{k}$  and
$\{c^{(k)}_{j}\}_{j=0}^{k}$ by the rules $w_0^{(0)} = 1$, $c_0^{(0)}=f[x_0]$ and
\begin{align*}
w_j^{(k)}& =\frac{ j }{k}w_{j-1}^{(k-1)}(1-x_{k-1}) - \frac{k-j}{k} w_j^{(k-1)} x_{k-1}, \\
c_j^{(k)} &=  \frac{j}{k}c_{j-1}^{(k-1)} + \frac{k-j}{k}c_{j}^{(k-1)} + w_j^{(k)}f[x_{0},...,x_{k}],
\end{align*}
for $j = 0,...,k$, where we use the convention $c_{-1}^{(k-1)}=w_{-1}^{(k-1)} =
0$ and $c_{k}^{(k-1)}=w_{k}^{(k-1)} = 0$.  Then, $\{w^{(k)}_{j}\}_{j=0}^{k}$ are
the control points of $w_k$ and $\{c^{(k)}_{j}\}_{j=0}^{k}$ are the control points of
$p_k$.
\end{theorem}
\begin{proof}
The case $k=0$ is trivially satisfied. We proceed by induction and suppose that
the Bernstein forms of the polynomials, $p_{k-1}$ and $w_{k-1}$ are given by
\begin{equation*}
	w_{k-1}(x) =\sum_{j=0}^{k-1}w_j^{(k-1)} B_j^{k-1}(x) \mbox{ and }
	p_{k-1}(x) =\sum_{j=0}^{k-1}c_j^{(k-1)} B_j^{k-1}(x).
\end{equation*}
Firstly, we derive the  Bernstein coefficients of the $k$-th degree polynomial
$w_k$ as follows:
\begin{align*}
  w_k(x) & =(x-x_k)  \sum_{j=0}^{k-1} w^{(k-1)}_j B^{k-1}_j(x) \\
         & = \left(x(1-x_{k-1})- x_{k-1}(1-x) \right) \sum_{j=0}^{k-1} w^{(k-1)}_j B^{k-1}_j(x).
\end{align*}
Using~\eqref{eq:manuel} with $n=k-1$ we conclude that the Bernstein
coefficients of $w_k$ are given by
\begin{align}\label{BBformwk}
  w_j^{(k)}& =\frac{ j }{k}w_{j-1}^{(k-1)}(1-x_{k-1}) - \frac{k-j}{k} w_j^{(k-1)} x_{k-1}.
\end{align}
Secondly, we use degree raising property
\begin{equation}
B_{k}^{n-1}= \frac{n-k}{n} B_{k}^{n} + \frac{k+1}{n}B_{k+1}^{n},\quad k=0,..., n-1.
\end{equation}
to write the $(k-1)$-th degree polynomial  $p_{k-1}$, in terms of the Bernstein
basis of polynomials of degree $k$.
\begin{equation}\label{BBformpk-1}
p_{k-1}(x)=\sum_{j=0}^k c^{(k-1)}_j B^{k-1}_j(x)  = \sum_{j=0}^{k}\left(  \frac{j}{k}c_{j-1}^{(k-1)} + \frac{k-j}{k}c_{j}^{(k-1)}  \right)B_j^{k}(x)
\end{equation}
where we again use the convention $c_{-1}^{(k-1)} = 0$ and $c_{k}^{(k-1)} = 0$.
Observe from \eqref{p_k} that $p_k(x) = p_{k-1}(x)+w_{k}(x)f[x_{0},...,x_{k}]$.
Hence, by \eqref{BBformpk-1} we have that
 \begin{equation*}
 c_j^{(k)} =  \frac{j}{k}c_{j-1}^{(k-1)} + \frac{k-j}{k}c_{j}^{(k-1)} + w_j^{(k)}f[x_{0},...,x_{k}],\quad j = 0,...,k
\end{equation*}
are the control points of $p_k$.
\end{proof}

Algorithm~\ref{NB1} provides an implementation of the result obtained in
Theorem \ref{theoremNB} in which the simplicity of the procedure is immediately
apparent along with an overall complexity of $\mathcal{O}(n^2)$:
\begin{algorithm}
    \SetKwInOut{Input}{Input}
    \SetKwInOut{Output}{Output}
    \SetKwInOut{Comment}{Comment}
    \Input{Interpolation nodes and data $\{x_{j}\}_{j=0}^{n}$,$\{f_{j}\}_{j=0}^{n}$.}
    \Output{Control points $\{c_{j}\}_{j=0}^{n}$.}
    $c_0\gets f_{0}$\;
    $w_{0} \gets 1$\;
    \For{$k\gets n,s$}
    {   \For{$k\gets n,s$}
        {
            $f_{k} \gets (f_{k}-f_{k-1})/(x_{k}-x_{k-s})$
	    \tcp*{Divided differences}
        }
    	\For{$k\gets s,1$}
        {
            $w_{k}\gets \frac{k}{s}w_{k-1}(1-x_{s-1})-(1-\frac{k}{s})w_{k}x_{s-1}$\;
            $c_{k}\gets \frac{k}{s} c_{k-1}+(1-\frac{k}{s})c_{k}+f_{s}w_{k}$\;
        }
        $w_{0}\gets -w_0 x_{s-1}$\;
        $c_{0}\gets c_{0}+f_{s}w_{0}$\;
    }
    return $\{c_{j}\}_{j=0}^{n}$\;
    \caption{\texttt{NewtonBernstein} ($\{x_{j}\}_{j=0}^{n}$,$\{f_{j}\}_{j=0}^{n}$) }\label{NB1}
\end{algorithm}

Readers who have studied the derivation of the Marco-Martinez
algorithm~\cite{MR2305145}, for solving the same problem may be rather
surprised by the comparative ease with which the Newton-Bernstein algorithm is
derived. The complexity of both algorithms is $\mathcal{O}(n^2)$. In the next
section, we compare the performance of the Newton-Bernstein algorithm and the
more sophisticated Marco-Martinez algorithm.

\subsection{Numerical examples}\label{Numerics1D} In this section we compare
the performance of Algorithm~\ref{NB1} (\texttt{NewtonBernstein}) with
Marco-Martinez algorithm~\cite{MR2305145}  (\texttt{MM}) and, the command
``$\backslash$'' in \texttt{Matlab} for three examples.  Examples \ref{ex1} and
\ref{ex2} are taken directly from~\cite{MR2305145}. The true solutions of all
of these problems are computed using the command \texttt{linsolve} in
\texttt{Maple 18}.  In each case, we present the relative error defined by 
\begin{equation}\label{Relativeerror}
\textrm{Relative error} = \frac{\|c_{\textrm{exact}} - c_{\textrm{approx}}\|_2}{\|c_{\textrm{exact}} \|_2},
\end{equation}
where $c_{\textrm{exact}}$ and $c_{\textrm{approx}}$ denote the exact and
approximate B\'ezier control points.
\begin{example}\label{ex1}\normalfont
Let $n=15$ and choose uniformly distributed nodes given by $x_{i} =
(i+1)/(17)$, $i=0,...,n$.  Let $A$ be the Bernstein-Vandermonde matrix of
degree $n$ generated by a given set of interpolation nodes
$\{x_{j}\}_{j=0}^{n}$, 
\begin{equation}
a_{i,j} =  b_{i}^{n}(x_{j}) =  \binom{n}{i} (1-x_{j})^{n-i} x_{j}^{i}, \quad i,j=0,...,n.
\end{equation}

In this case the condition number of the Bernstein-Vandermonde matrix is
$\kappa(A) = 2.3e+06$. The right hand sides $f_1,\,f_2$ and $f_3$ are given by
\begin{align*}
(f_1)_{j}&=(1-x_j)^n, \\
f_2&= (2, 1, 2, 3, -1, 0, 1, -2, 4, 1, 1, -3, 0, -1, -1, 2)^{\mathtt{T}},\\
f_3&= (1, -2, 1, -1, 3, -1, 2, -1, 4, -1, 2, -1, 1, -3, 1, -4)^{\mathtt{T}}.
\end{align*}
Relative errors for each algorithm are displayed in Table \ref{tab:ex1}.

\begin{table}[htbp]
  \centering
  \begin{tabular}{c  @{\hskip .7in}  c @{\hskip .7in}   c@{\hskip .7in}    c }
    \toprule
    $f_i$&  $A\backslash f_i$  & \texttt{NewtonBernstein}  & \texttt{MM}\\
    \midrule
    $f_1$ &   7.2e-13& 7.9e-14  & 9.2e-13 \\
    $f_2$ & 7.1e-11  & 5.9e-16  & 1.0e-15  \\
    $f_3$ & 7.1e-11  & 5.2e-16  &   4.9e-16\\
 \bottomrule
  \end{tabular}\vskip2mm
  \caption{Example \ref{ex1}: Relative errors in $L^2$ norm.}
  \label{tab:ex1}
\end{table}
The results in Table \ref{tab:ex1} indicate that Newton-Bernstein and
Marco-Martinez algorithm give comparable stability and accuracy, with the 
``$\backslash$'' solver performing poorly.
\end{example}

\begin{example}\label{ex2}\normalfont
Let $n = 15$ and $A$ be the Bernstein-Vandermonde matrix generated by the data
\begin{equation}
x=(\frac{1}{18}, \frac{1}{16}, \frac{1}{14}, \frac{1}{12}, \frac{1}{10}, \frac{1}{8}, \frac{1}{6}, \frac{1}{4}, \frac{11}{20}, \frac{19}{34}, \frac{17}{30}, \frac{15}{26}, \frac{11}{18}, \frac{9}{14}, \frac{7}{10}, \frac{5}{6})^{\mathtt{T}}.
\end{equation}
The condition number of the Bernstein-Vandermonde matrix is $\kappa(A) =
3.5e+09$.  Consider the singular value decomposition of the
Bernstein-Vandermonde matrix $A = U \Sigma V^{\mathtt{T}}$. We will solve the
linear systems $Ac = f_{j}$, with $f_{j} = u_{j},\,j = 0,...,n,$ where $u_{j}$
denotes the $j-th$ column of $U$, the left singular vectors of $A$. Relative
errors for each algorithm are displayed in Table \ref{tab:ex2}.

{\small
\begin{table}[htbp]
  \centering
  \begin{tabular}{l  @{\hskip .5in}  l @{\hskip .5in}  c @{\hskip .5in}   c@{\hskip .5in}    c }
    \toprule
    $f_i$& $\gamma(A,f_i)$ & $A\backslash f_i$  & \texttt{NewtonBernstein}&  \texttt{MM}\\
    \midrule
    $f_1$   & 3.5e+09 & 3.3e-07  & 1.9e-08  & 3.1e-07\\
    $f_2$   & 2.6e+09 & 1.1e-07  & 6.2e-08  & 2.1e-08\\
    $f_3$   & 1.3e+09 & 8.4e-09  & 5.6e-09  & 2.7e-08\\
    $f_4$   & 1.2e+09 & 2.1e-08  & 1.1e-08  & 1.4e-08\\
    $f_5$   & 5.3e+08 & 3.8e-08  & 2.6e-09  & 1.7e-08\\
    $f_6$   & 4.0e+08 & 4.3e-08  & 1.0e-08  & 2.7e-09\\
    $f_7$   & 1.1e+08 & 3.7e-08  & 1.8e-09  & 7.4e-09\\
    $f_8$   & 5.8e+07 & 3.5e-08  & 6.5e-10  & 2.5e-09\\
    $f_9$   & 1.1e+07 & 1.5e-08  & 8.7e-10  & 3.2e-10\\
    $f_{10}$& 3.7e+06 & 1.2e-08  & 1.5e-10  & 3.3e-10\\
    $f_{11}$& 4.8e+05 & 1.4e-08  & 4.5e-12  & 3.6e-12\\
    $f_{12}$& 1.2e+05 & 3.5e-09  & 1.3e-11  & 1.7e-11\\
    $f_{13}$& 6.2e+03 & 8.2e-09  & 3.0e-12  & 2.3e-12\\
    $f_{14}$& 9.3e+02 & 1.1e-08  & 7.6e-13  & 7.8e-13\\
    $f_{15}$& 1.4e+01 & 1.2e-08  & 4.2e-14  & 4.2e-14\\
    $f_{16}$& 1       &  5.0e-09 & 7.1e-15  & 7.9e-15\\
 \bottomrule
  \end{tabular}\vskip2mm
  \caption{Example \ref{ex2}: Relative errors in $L^2$ norm.}
  \label{tab:ex2}
\end{table}
}
This example illustrates the effective well-conditioning introduced by Chan and
Foulser in \cite{MR963849}. In particular, the test confirms that the Newton-Bernstein
algorithm increases in accuracy as the Chan-Foulser number decreases which is the
same (positive) behaviour exhibited by the Marco-Martinez algorithm.  
\end{example}

\begin{example}\label{ex3}\normalfont
Let $n=25$ and take the interpolation nodes to be the zeros of the Chebyshev
polynomial of the first kind. In this case the condition number of the
Bernstein-Vandermonde matrix is $\kappa(A) = 2.1e+07$. The right hand sides are 
taken to be $f_1,\,f_2$ and $f_3$ given by
{\footnotesize
\begin{align*}
(f_{1})_j&=(1-x_j)^n\quad\hbox{for } j =1,...,25, \\
f_2&=(-3, -1, 2, -1, 2, -1, 1, -3, 2, -3, 1, 2, -1, -2, 1, -2, -1, -2, 1, -2, 3, -2, -3, 2, 1, -2)^{\mathtt{T}},\\
f_3 &=(-1, 2, 1, -1, -2, -3, 2, 3, -2, -1, 2, 1, 3, -2, 1, -1, -1, 2, -2, -3, 1, -1, 1, -3, 2, -1)^{\mathtt{T}}.
\end{align*}
}
Relative errors for each algorithm are displayed in Table \ref{tab:ex3}.

\begin{table}[htbp]
  \centering
  \begin{tabular}{c  @{\hskip .4in}  c @{\hskip .4in}   c@{\hskip .2in}    c@{\hskip .4in}    c }
    \toprule
    $f_i$&  $A\backslash f_i$  & \texttt{NewtonBernstein}$_{\mathrm{Leja}}$ & \texttt{NewtonBernstein} & \texttt{MM}\\
    \midrule
    $f_1$ &  7.7e-11& 4.2e-11  &  4.2e-11 &  1.0e-09 \\
    $f_2$ &  1.1e-10 &  3.2e-16 &   7.9e-13& 1.4e-15 \\
    $f_3$ &  1.1e-10  &  4.8e-16 &  1.6e-13  &   1.6e-15\\
 \bottomrule
  \end{tabular}\vskip2mm
  \caption{Example \ref{ex3}: Relative errors in $L^2$ norm.}
  \label{tab:ex3}
\end{table}
It is well-known that there can be advantages in re-ordering the interpolation
nodes. One feature of the Newton-Bernstin Algorithm (not shared by the
Marco-Martinez Algorithm\texttt{MM}) is the flexibility to re-order the
interpolation nodes. Table \ref{tab:ex2} compares the relative errors obtained 
for the interpolation problems using ascending ordering (\texttt{NewtonBernstein}) and
using Leja ordering \texttt{NewtonBernstein}$_{\mathrm{Leja}}$, see
\cite{MR1039671,MR1927606}. In this example the Leja ordering of
the interpolation nodes produces better results, and as in the previous example
the accuracy of our algorithms increases with the alternating sign pattern of
the right hand side. 
\end{example}

\section{Tensor product polynomial interpolation}
In this section we briefly show how the Newton-Bernstein Algorithm~\ref{NB1}
can be used to interpolate on tensor product grids in higher dimensions.
Consider the two dimensional case. Let $\{x_{i}\}_{i=0}^{n}\subseteq [0,1]$ and
$\{y_{j}\}_{j=0}^{m}\subseteq[0,1]$ be given nodes and let
$\{f_{i,j}\}_{i,j=0}^{n,m}$ be given data. The polynomial interpolation problem
consists of finding the polynomial
$p\in\mathbb{P}^{n}([0,1])\times\mathbb{P}^{m}([0,1])$ satisfying the
conditions:
\begin{equation}\label{polyfit2D}
p(x_i,y_j) =  f_{i,j}, \qquad i=0,...,n, j =0,...,m.
\end{equation}
The associated Bernstein interpolation problem consists of finding the control points
$\{c_{k,\ell}\}_{k=0,\ell=0}^{n,m}$ such that the tensor product Bernstein polynomial
\begin{equation}\label{Bernsteinrepresentation2D}
p(x,y)  =  \sum_{\ell = 0}^{m} \sum_{k=0}^n c_{k,\ell} B_k^{n}(x) B_{\ell}^{m}(y),\quad (x,y)\in[0,1]^2
\end{equation}
satisfies \eqref{polyfit2D}.

The tensor product nature of the problem means that we can exploit the
Newton-Bernstein algorithm for the univariate case to solve
problem~\eqref{Bernsteinrepresentation2D}.  The basic idea is to: a) first
construct the control points for the univariate Bernstein interpolant $p^{(j)}$
on the lines $y=y_j$ for each $j$, i.e.  for each $j=0,\ldots,n$:
\begin{equation}\label{pjsol}
   p^{(j)}(x)=\sum_{k=0}^{n}c_{k}^{(j)}B_{k}^{n}(x);\quad
   p^{(j)}(x_{i}) = f_{i,j},\quad i = 0,...,n;
\end{equation}
and, b) solve a univariate interpolation problem for the $y$-variable in which
the data are the univariate polynomials obtained in step a), i.e. find the
univariate polynomial $p$ satisfying
\begin{equation}
p(x,y_{j}) = p^{(j)}(x),\quad j = 0,...,m.
\end{equation}

The problem in step b) is nothing more than a univariate interpolation problem
with the only difference being that the data is now polynomial valued.  The
derivation of the Newton-Bernstein algorithm presented in the previous section
applies equally well to the more general problem of interpolating values from a
vector space $X$.  In particular, choosing $X$ to be polynomials shows that the
Newton-Bernstein Algorithm~\ref{NB1} can be brought to bear at each of stages
a) and b). This idea forms the foundation for Algorithm~\ref{TP2D}.
\begin{algorithm}
    \SetKwInOut{Input}{Input}
    \SetKwInOut{Output}{Output}
    \SetKwInOut{Comment}{Comment}
    \Input{Interpolation nodes and data $\{x_{i}\}_{i=0}^{n}$, $\{y_{j}\}_{j=0}^{m}$, $\{f_{i,j}\}_{i,j=0}^{n,m}$.}
    \Output{Control points $\{c_{j}\}_{j=0}^{n}$.}
    \For{$j\gets 0,m$}
    {
        $\{\tilde{c}^{(j)}_k\}_{k=0}^{n}\gets$ \texttt{NewtonBernstein}$(\{x_{i}\}_{i=0}^{n},\{f_{i,j}\}_{i=0}^{n})$\;
    }
    \For{$i\gets 0,n$}
    {
        $\{c_{i,j}\}_{j=0}^{m} \gets$ \texttt{NewtonBernstein}$(\{y_{j}\}_{j=0}^{m},\{\tilde{c}_{i}^{(j)}\}_{j=0}^{m})$\;
    }
    return $\{c_{i,j}\}_{i,j=0}^{n,m}$\;
    \caption{\texttt{TensorProduct2D} ($\{x_{i}\}_{i=0}^{n}$, $\{y_{j}\}_{j=0}^{m}$, $\{f_{i,j}\}_{i,j=0}^{n,m}$) }\label{TP2D}
\end{algorithm}

The basic idea used in the two dimensional algorithm is easily extended to
higher dimensions resulting in Algorithm~\ref{TP3D} which computes the solution
of the three dimensional version of the interpolation problem
\eqref{Bernsteinrepresentation2D}.
\begin{algorithm}
    \SetKwInOut{Input}{Input}
    \SetKwInOut{Output}{Output}
    \SetKwInOut{Comment}{Comment}
    \Input{Interpolation nodes and data $\{x_{i}\}_{i=0}^{n}$, $\{y_{j}\}_{j=0}^{m}$,$\{z_{k}\}_{k=0}^{l}$, $\{f_{i,j,k}\}_{i,j,k=0}^{n,m,l}$.}
    \Output{Control points $\{c_{i,j,k}\}_{i,j,k=0}^{n,m,l}$.}
    \For{$j\gets 0,m$}
    {
        $\{\tilde{c}_{i,j,k}\}_{i=0}^{n}\gets$\texttt{NewtonBernstein}$(\{x_{i}\}_{i=0}^{n},\{f_{i,j,k}\}_{i=0}^{n})$\;
    }
    \For{$i\gets 0,n$}
    {
        $\{\hat{c}_{i,j,k}\}_{j=0}^{m} \gets$ \texttt{NewtonBernstein}$(\{y_{j}\}_{j=0}^{m},\,\{\tilde{c}_{i,j,k}\}_{j=0}^{m})$\;
    }
    \For{$i\gets 0,n$}
    {
	    $\{c_{i,j,k}\}_{k=0}^{l} \gets$ \texttt{NewtonBernstein}$(\{z_{k}\}_{k=0}^{l},\,\{\hat{c}_{i,j,k}\}_{k=0}^{l})$\;
	}
    return $\{c_{i,j,k}\}_{i,j,k=0}^{n,m,l}$\;
    \caption{\texttt{TensorProduct3D} ($\{x_{i}\}_{i=0}^{n}$, $\{y_{j}\}_{j=0}^{m}$,$\{z_{k}\}_{k=0}^{l}$, $\{f_{i,j,k}\}_{i,j,k=0}^{n,m,l}$) }\label{TP3D}
\end{algorithm}

\subsection{Numerical examples: Tensor product}
In this section we consider examples of the two and three dimensional Bernstein-B\'ezier interpolation problems, illustrating the performance of  Algorithms~\ref{TP2D} and~\ref{TP3D}.
\begin{example}\label{ex4}\normalfont
Consider $n = 15$ and  the two dimensional interpolation problem with grid
induced by the following nodes
\begin{align*}
x_{i} = \frac{i+1}{n+2} ,\quad y_{j} = \frac{j+1}{n+3}\quad i,j = 0,...,n.
\end{align*}
Note that in this case the condition number of the two dimensional Bernstein-\\
Vandermonde matrix $A_{2}$ (size $256\times256$) is $\kappa(A_2) = 1.4e+13$. As
load vectors we consider $f_1$ and  $f_2$ randomly generated, taking integer
values between $-3$ and $3$ for each component.  Relative errors are displayed
in Table \ref{tab:ex4}.

\begin{table}[htbp]
  \centering
  \begin{tabular}{c  @{\hskip .6in}  c @{\hskip .2in} c @{\hskip .3in}  c@{\hskip .5in}  c  @{\hskip .5in}   c   }
    \toprule
   & \multicolumn{1}{l}{} & & \multicolumn{3}{c}{SOLVER2D} \\
    \cmidrule{4-6}
    $f_i$&   $A_2\backslash f_{i}$  & &\texttt{NewtonBernstein}  &$A\backslash f_{i}$  & \texttt{MM} \\
    \midrule
    $f_{1}$ &4.2e-5  & &  2.5e-15& 2.3e-11   &  1.6e-15\\
    $f_{2}$ &2.6e-5  & &  9.7e-16& 3.5e-11   &  2.8e-15\\
 \bottomrule
  \end{tabular}\vskip2mm
  \caption{Relative errors in $L^2$ norm for solutions of Example \ref{ex4} using \texttt{Matlab} ``$\backslash$'' and \texttt{TensorProduct2D}. The results obtained when \texttt{NewtonBernstein} is replaced in \texttt{TensorProduct2D} by the univariate solvers ``$\backslash$'' and \texttt{MM} are also presented.}
  \label{tab:ex4}
\end{table}
\end{example}
Observe the significantly higher accuracy of the Newton-Bernstein algorithm in comparison 
with the results of ``$\backslash$''.

\begin{example}\label{ex5}\normalfont
Consider $n = 10$ and  the three dimensional interpolation problem with grid induced for the following nodes
\begin{align*}
x_{i} = \frac{i+1}{n+2} ,\quad y_{j} = \frac{j+1}{n+3},\quad z_{k} = \frac{k+2}{n+4}\quad i,j = 0,...,n.
\end{align*}

Note that in this case the condition number of the three dimensional  Bernstein-\\ Vandermonde matrix (size $1331\times1331$) is $\kappa(A_3) = 7.6e+13$. As load vectors we consider $f_1$ and  $f_2$  randomly generated, taking integer values between $-3$ and $3$ for each component.  Relative errors are displayed in Table \ref{tab:ex5}.

\begin{table}[htbp]
  \centering
  \begin{tabular}{c  @{\hskip .6in}  c @{\hskip .2in} c @{\hskip .3in}  c@{\hskip .5in}  c  @{\hskip .5in}  c   }
    \toprule
   & \multicolumn{1}{l}{} & & \multicolumn{3}{c}{SOLVER3D} \\
    \cmidrule{4-6}
    $f_i$&   $A_3\backslash f_{i}$ & &   \texttt{NewtonBernstein}&$A\backslash f_{i}$   & \texttt{MM} \\
    \midrule
    $f_1$ & 8.4e-6 & &  6.0e-16&  2.1e-12  &   1.1e-15\\
    $f_2$ & 8.2e-6 & &  5.2e-16&  2.2e-12  &   1.3e-15\\
 \bottomrule
  \end{tabular}\vskip2mm
  \caption{Relative errors in $L^2$ norm for solutions of Example \ref{ex5} using \texttt{Matlab} ``$\backslash$'' and \texttt{TensorProduct3D}. The results obtained when \texttt{NewtonBernstein} is replaced in \texttt{TensorProduct3D} by the univariate solvers ``$\backslash$'' and \texttt{MM} are also presented.}
  \label{tab:ex5}
\end{table}

The three dimensional case also shows a great accuracy of the Newton-Bernstein algorithm in comparison with the \texttt{Matlab} solver, see Table \ref{tab:ex5}.
\end{example}

\section{Control points for polynomial interpolation on a simplex}\label{section simplex}
The computation of the control points of the Lagrange interpolant on a simplex
in $\mathbb{R}^d$, $d\in\mathbb{N}$ is rather more problematic than the tensor
product case. Nevertheless, in this section, we show how the univariate Newton-Bernstein
algorithm can be generalised to solve the problem.

\subsection{Preliminaries}
For $n\in \mathbb{N}$ define the indexing set
$\Idn=\{ \n{\alpha}= (\alpha_{1},\ldots,\alpha_{d+1})\in \mathbb{Z}_*^{d+1} :
|\n{\alpha}|=n\}$ where $|\n{\alpha}|=\sum_{j=1}^{d+1}\alpha_j$,
$\mathbb{Z}_{*} = \{0\}\cup\mathbb{N}$ and $|\Idn| = \mathrm{card}(\Idn) =
\binom{n+d}{d}$.  For $n,m\in\mathbb{N}$, $\n{\alpha}\in\Idn$,
$\n{\beta}\in\mathcal{I}_{d}^m$ and $\n{\lambda}\in\mathbb{R}^{d+1}$, define
\begin{equation*}
\binom{\n{\alpha}}{\n{\beta}} = \prod_{j=1}^{d+1} \binom{\alpha_k}{\beta_k},
\quad\binom{n}{\n{\alpha}}=n! \,\Big(\prod_{j=1}^{d+1}\alpha_{j}!\Big)^{-1},
\quad \n{\lambda}^{\n{\alpha}}=\prod_{j=1}^{d+1} \lambda_{j}^{\alpha_j}.
\end{equation*}

Let $T=conv\{\n{v}_1,...,\n{v}_{d+1}\}$ be a non-degenerate $d$-simplex in $\mathbb{R}^{d}$.
The following property of a non-degenerate simplices will prove useful~\cite{MR3183046}:
\begin{lemma}\label{lemmaBC}
The following conditions are equivalent:
\begin{itemize}
\item[1.] $T$ is a non-degenerate $d$-simplex;
\item[2.] for all $\n{x}\in T$, there exists a unique set of nonnegative
	scalars $\lambda_1,\,\lambda_2,...,\,\lambda_{d+1}$ such that
	$\sum_{\ell=1}^{d+1}\lambda_l\n{v}_\ell=\n{x}$ and $\sum_{\ell=1}^{d+1}
	\lambda_{\ell} =1$.
\end{itemize}
\end{lemma}

Let $\mathcal{D}_{d}^{n}$ be the set of domain points of $T$ defined by
$\n{x}_{\n{\alpha}}=\frac{1}{n}\sum_{k=1}^{d+1}\alpha_k\n{v}_k$, $\n{\alpha}\in
\Idn$, where $\n{\lambda}\in \mathbb{R}^{d+1}$ are defined as in
Lemma~\ref{lemmaBC}.  The Bernstein polynomials of degree $n\in \mathbb{Z}_+$
associated with the simplex $T$ are defined by
\begin{equation}\label{}
    B_{\n{\alpha}}^{T,n}(\n{x}) =B_{\n{\alpha}}^n(\n{x})
    = \binom{n}{\n{\alpha}} \n{\lambda}(\n{x})^{\n{\alpha}},\qquad \n{\alpha}\in \Idn
\end{equation}
and satisfy
\begin{equation}\label{dprod}
   B_{\n{\alpha}}^{m} B_{\n{\beta}}^{n}
   = \dfrac{\binom{\n{\alpha}+\n{\beta}}{\n{\alpha}}}{\binom{m+n}{n}}B_{\n{\alpha}+\n{\beta}}^{m+n}
   \qquad\n{\alpha}\in \mathcal{I}^{m}_{d},\, \n{\beta}\in \Idn.
\end{equation}

The Bernstein polynomial interpolation problem on the simplex $T$ reads: given a set of
distinct interpolation nodes $\big\{\n{x}_{j}\big\}_{j=1}^{|\Idn|}$ and
interpolation data  $\big\{f_{j}\big\}_{j=1}^{|\Idn|}$, find control points
$\{c_{\n{\alpha}}\}_{\n{\alpha}\in \Idn}$ such that
\begin{equation}\label{Simplexinterp}
p(\n{x})=\sum_{\n{\alpha}\in \Idn} c_{\n{\alpha}} B_{\n{\alpha}}^{n}(\n{x}):
\quad p(\n{x}_{j})  =  f_{j},\quad \hbox{for } j= 1,...,|\Idn|.
\end{equation}
\subsection{Two dimensional case}
The polynomial interpolation problem in higher dimensions requires that the
interpolation nodes appearing in~\eqref{Simplexinterp} satisfy additional
conditions beyond simply being distinct in order for the interpolation problem
to be well-posed~\cite{MR892779}:
\begin{description}
\item [Solvability Condition (S)]
	\emph{There exist distinct lines $\gamma_0,\gamma_1,\ldots,\gamma_n$
	such that the interpolation nodes can be partitioned into
	(non-overlapping) sets $\cA_{n}$,$\cA_{n-1},\ldots$, $\cA_{0}$ where $\cA_j$
	contains $j+1$ nodes located on $\gamma_j\backslash
(\gamma_{j+1},...,\gamma_n)$. }
\end{description}
\begin{figure}
  \begin{minipage}[c]{0.4\textwidth}
    \includegraphics[width=\textwidth]{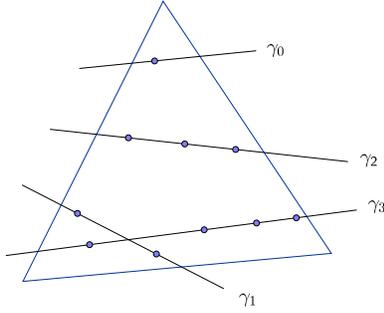}
  \end{minipage}\hfill
  \begin{minipage}[c]{0.6\textwidth}
    \caption{Illustration of configuration of interpolation nodes under
    the Solvability Condition (S)} \label{figureAssumption}
  \end{minipage}
\end{figure}
We will assume that Condition (S) is satisfied throughout. In particular,
Condition (S) implies that there is a line $\gamma_n$ on which $n+1$ distinct
interpolation nodes lie. This means that there exists a (non-unique) polynomial
$q_n\in\mathbb{P}^{n}(T)$ satisfying the \emph{univariate} polynomial
interpolation problem on the line $\gamma_n$:
\begin{equation}
   q_{n}(\n{x}_{i})=f_{i} \quad\forall \n{x}_{i}\in\mathcal{A}_n.
\end{equation}
Equally well, Condition (S) implies that the line $\gamma_{n-1}$ contains $n$
distinct interpolation nodes, none of which lie on $\gamma_n$, meaning that the
univariate polynomial interpolation problem for these nodes is well-posed. The
data for this interpolation problem is chosen as follows. Let $\Gamma_n$ be an
affine polynomial describing the line $\gamma_n$, i.e.  $\n{x}\in\gamma_n$ iff
$\Gamma_{n}(\n{x})=0$. There exists a (non-unique) polynomial
$q_{n-1}\in\mathbb{P}^{n-1}(T)$ satisfying the following \emph{univariate}
polynomial interpolation problem on the line $\gamma_{n-1}$:
\begin{equation}
q_{n-1}(\n{x}_{i})=\frac{f_{i} - q_{n}(\n{x}_{i})\Gamma_{n}(\n{x}_{i})}{\Gamma_{n}(\n{x}_{i})}
\quad \forall  \n{x}_{i}\in\mathcal{A}_{n-1}.
\end{equation}
Observe that this data is well-defined thanks to Condition (S). The
non-uniqueness of the polynomial stems from the fact that while the values of
$q_n$ on the line $\gamma_n$ are uniquely defined, there are many ways to
extend $q_n$ to the simplex---a canonical approach for defining the extension
will be presented in Section~\ref{sectionBB-q}.

The foregoing arguments can be applied repeatedly to define a sequence of
polynomials $q_j\in\mathbb{P}^{j}(T)$ satisfying a univariate interpolation
problem on the line $\gamma_j$:
\begin{equation}\label{defqj}
q_{j}(\n{x}_{i})=\frac{f_{i} - \sum_{k=j+1}q_{k}(\n{x}_{i})
\prod_{l=k+1}^{n}\Gamma_{l}(\n{x}_{i})}{
\prod_{l=j+1}^{n}\Gamma_{l}(\n{x}_{i})} \quad \forall
\n{x}_{i}\in\mathcal{A}_j
\end{equation}
where $\Gamma_k$ is an affine polynomial satisfying $\n{x}\in\gamma_k$ iff
$\Gamma_{k}(\n{x})=0$, $k=0,\ldots,n$.

The following result presents a general construction for the solution of the full
interpolation problem~\eqref{Simplexinterp} in terms of solutions $q_j$ of
univariate interpolation problems on the lines $\gamma_j$:
\begin{theorem}\label{Simplex2Dtheo}
Let $\{q_j\}_{j=0}^n$ be defined as above and define $p\in\mathbb{P}^{n}(T)$ to be
\begin{equation}\label{Newtonform2D}
p(\n{x}) = \sum_{j=0}^{n} q_{j}(\n{x}) \prod_{i=j+1}^{n} \Gamma_{i}(\n{x}), \quad\n{x}\in T.
\end{equation}
Then $p$ solves the interpolation problem \eqref{Simplexinterp},
\end{theorem}
\begin{proof}
The polynomial $p$ defined in~\eqref{Newtonform2D} clearly belongs to
$\mathbb{P}^n(T)$.  Let $j\in\{0,..,n\}$ and $\n{x}_{i}\in\mathcal{A}_j$,
$j\in\{0,...,n\}$. Inserting $\n{x}_i$ into formula~\eqref{Newtonform2D} gives
\begin{align*}
p(\n{x}_{i}) &=\sum_{k=0}^{n} q_{k}(\n{x}_{i}) \prod_{l=k+1}^{n} \Gamma_{l}(\n{x}_{i}) =\sum_{k=j}^{n} q_{k}(\n{x}_{i}) \prod_{l=k+1}^{n} \Gamma_{l}(\n{x}_{i}) \\
&= q_{j}(\n{x}_{i}) \prod_{l=j+1}^{n}\Gamma_{l}(\n{x}_{i}) + \sum_{k=j+1}^{n} q_{k}(\n{x}_{i}) \prod_{l=k+1}^{n} \Gamma_{l}(\n{x}_{i}),
\end{align*}
where we have used the fact that $\Gamma_j(\n{x}_i)=0$ due to Condition (S).
The result follows thanks to \eqref{defqj}.
\end{proof}

Theorem~\ref{Simplex2Dtheo} reduces the solution of the multivariate
interpolation problem to the solution of a sequence of univariate interpolation
problems. This feature may be used in conjunction with the univariate
Newton-Bernstein algorithm to construct an algorithm for solving the
multivariate Bernstein interpolation problem~\eqref{Simplexinterp} as follows:

\begin{algorithm}
\SetKwInOut{Input}{Input}
    \SetKwInOut{Output}{Output}
    \SetKwInOut{Comment}{Comment}
    \Input{Interpolation nodes and data $\left\{\n{x}_j, \,f_j\right\}_{j=0}^{|\cI^{n}_{2}|}$ .}
    \Output{Control points $c^{p}$.}
    Initialise $c^{p}$\;
    \For{$j\gets n,1$}
    {
        $G^{j}\gets$ \texttt{BBAffine}$(\mathcal{A}_{j})$\tcp*{BB-form of $\Gamma_{j}$}
        $[\n{z}_1,\n{z}_2,\kappa]\gets \texttt{GcapT}(G^{j})$\tcp*{$\Gamma_j\cap T$}
        $\bar{\mathcal{A}}_j\gets \texttt{Transform1D}(\mathcal{A}_j,\n{z}_1,\n{z}_2)$\tcp*{Map $[\n{z}_1,\n{z}_2]\mapsto [0,1]$}
        $c^{\gamma_{j}}\gets$ \texttt{NewtonBernstein}$(\bar{\mathcal{A}}_{j})$\;
        $c^{q_{j}}\gets$ \texttt{BBExtension}$(c^{\gamma_{j}},\n{z}_1,\n{z}_2,\kappa)$ \tcp*{$q_{j}$ solves \eqref{defqj}}
        $a\gets c^{q_{j}}$\;
        \tcc{Computes BB-form of $q_{j}\prod_{i=j+1}^{n} \Gamma_{i}$}
        \For{$i\gets j+1,n$}
        {
            $a \gets$ \texttt{BBProduct}$(a, G^{i})$\;
        }
        $c^{p} \gets c^{p}+a$\;
        \For{$i\gets 1,|\cI^{j-1}_{2}|$}
        {
        \tcc{Divided differences}
            $f_{i} \gets \Big(f_{i} -$\texttt{DeCasteljau}$(c^{q_{j}},\n{x}_{i} )\Big)/$ \texttt{DeCasteljau}$(G^{j},\n{x}_{i})$\;
        }      
    }
    return $c^{p}$\;
\caption{\texttt{NewtonBernstein2D}($\left\{\n{x}_j, \,f_j\right\}_{j=0}^{|\cI^{n}_{2}|}$)}\label{SS2D}
\end{algorithm}

Algorithm~\ref{SS2D} calls five subroutines.  The subroutine \texttt{DeCasteljau}
refers to the well-known \emph{de Casteljau}
algorithm~\cite{Farin:2001:CSC:501891} for the evaluation of a polynomial
written in Bernstein form at a cost of $\cO(n^{d+1})$ operations in
$d$-dimensions. The four remaining subroutines are the subject of the
following sections. It will transpire that the overall complexity of Algorithm
\ref{SS2D} is $\mathcal{O}(n^{d+1})$, where $d=2$ in the current case.
A numerical example illustrating the performance of the algorithm is presented
in Section \ref{Numericssimplex}.

\begin{algorithm}
    \SetKwInOut{Input}{Input}
    \SetKwInOut{Output}{Output}
    \SetKwInOut{Comment}{Comment}
    \Input{Interpolation nodes and data $\cA_j$ .}
    \Output{BB-form of $\Gamma$, $G$.}
    Compute coefficients $a$, $b$, $c$ of line $ax+by + c= 0$ fitting the interpolation nodes $\mathcal{A}_j$.\;  
    $d = \max\{|c|,|a+c|,|b+c| \}$\;
    $G_{(1,0,0)} = c/d$\;
    $G_{(0,1,0)} = (a+c)/d$\;
    $G_{(0,0,1)} = (b+c)/d$\;
    return $G$\;
\caption{\texttt{BBAffine}( $\cA_j$)}\label{BB-Gamma}
\end{algorithm}

\subsubsection{\texttt{BBExtension}}\label{sectionBB-q}
The subroutine \texttt{BBExtension} extends the domain of a solution of the
univariate interpolation problem~\eqref{defqj} from the line $\gamma_j$ to the
whole simplex $T$. In order to accomplish this task, it is necessary to obtain
points $\n{z}_1$ and $\n{z}_2$ satisfying $\gamma_j\cap T =
conv\{\n{z}_1,\n{z}_2\}$. Roughly speaking, $\n{z}_1$ and $\n{z}_2$ are the
points at which the line $\gamma_j$ intersects the boundary of the simplex.
Without loss of generality, assume that $\gamma_j$ separates the vertices of
$T$ into sets $\{\n{v}_1,\n{v}_2\}$ and $\{\n{v}_3\}$ such that
$\lambda_1(\n{z}_1)>0$,  $\lambda_1(\n{z}_2) = 0$ and $\lambda_2(\n{z}_1)=0$,
$\lambda_{2}(\n{z}_2)>0$, where $\lambda_1$ and $\lambda_2$ are the barycentric
coordinates on the simplex $T$.
\begin{figure}
  \begin{center}
    \includegraphics[width=\textwidth]{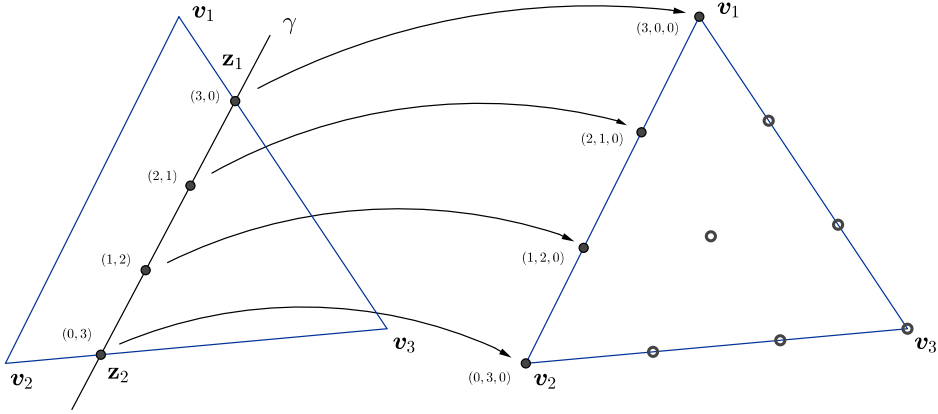}
    \caption{Illustration of labeling for extension procedure and
    correspondence of control points
    $\{c^{\gamma}_{\n{\alpha}}\}_{\n{\alpha}\in\cI_1^j}$ and
    $\{c^{T}_{\n{\alpha}}\}_{\n{\alpha}\in\cI_2^j}$ (filled circles), with
    $j=3$. } \label{figurePyra}
  \end{center}
\end{figure}

In particular, $conv\{\n{z}_1,\n{z}_2\}$ is a non-degenerate 1-simplex and
hence, in view of Lemma~\ref{lemmaBC}, we can define barycentric coordinates on
the 1-simplex $conv\{\n{z}_1,\n{z}_2\}$.  In view of the theory presented in
Section~\ref{Newton1D}, we can find control points
$\{c_{\n{\alpha}}^{\gamma_j}\}$ of the univariate Bernstein polynomial
interpolant of the data $\cA_j$ on $conv\{\n{z}_1,\n{z}_2\}$. These control
points correspond to the Bernstein form of the interpolant relative to the
barycentric coordinates on the 1-simplex $conv\{\n{z}_1,\n{z}_2\}$.  It
therefore only remains to transform these control points to control points for
the barycentric coordinates on the simplex $T$, and thereby implicitly
extending the univariate polynomial to the whole simplex:
\begin{lemma}\label{Lemmaextension}
Let $\{c^{\gamma_j}_{\n{\alpha}}\}_{\n{\alpha}\in\cI^{j}_1}$ be the control
points of a polynomial $q^{\gamma_j}\in
\mathbb{P}^{j}(conv\{\n{z}_1,\n{z}_2\})$ defined on the 1-simplex
$conv\{\n{z}_1,\n{z}_2\}$. Define control points
$\{c^{T}_{\n{\alpha}}\}_{\n{\alpha}\in\cI^{j}_2}$ of a polynomial $q^{T}\in
\mathbb{P}^{j}(T)$ on the simplex $T$ by the rule: for
$(\alpha_1,\alpha_2,\alpha_3)\in\cI^{j}_{2}$
\begin{equation}\label{cT}
  c^{T}_{(\alpha_1,\alpha_2,\alpha_3)}= \left\{
     \begin{array}{ll}
       c^{\gamma_j}_{(\alpha_1,\alpha_2)}(\lambda_1(\n{z}_1))^{-\alpha_1}(\lambda_2(\n{z}_2))^{-\alpha_2}, & \hbox{if } \alpha_3 = 0; \\
       0, & \hbox{otherwise.}
     \end{array}
   \right.
\end{equation}
Then, $q^T\in\mathbb{P}^{j}(T)$ coincides with $q^{\gamma_j}$ on $\gamma_j$.
\end{lemma}
\begin{proof}
Consider the Bernstein form of $q^T$ and apply definition \eqref{cT}
\begin{align}
\nonumber q^{T}(\n{x})&= \sum_{\n{\alpha}\in\cI^{j}_2:\,\alpha_3=0} c^{T}_{\n{\alpha}} B_{\n{\alpha}}^{T,j}(\n{x})\\
                &=\sum_{\n{\alpha}\in\cI^{j}_2, \alpha_3=0} c^{\gamma_j}_{(\alpha_1,\alpha_2)} \binom{j}{(\alpha_1,\alpha_2)} \left(\frac{\lambda_1(\n{x})}{\lambda_1(\n{z}_1)}\right)^{\alpha_1} \left(\frac{\lambda_2(\n{x})}{\lambda_2(\n{z}_2)}\right)^{\alpha_2},\quad \n{x}\in T. \label{factor}
\end{align}
Observe the factor $\frac{\lambda_1(\n{x})}{\lambda_1(\n{z}_1)}$ is linear, and
takes values $1$ at $\n{x}=\n{z}_1$ and $0$ at $\n{x}=\n{z}_2$, and as such
corresponds to the barycentric coordinates on $conv\{\n{z}_1,\n{z}_2\}$.  The
same consideration apply to the other factor in \eqref{factor}. Hence, if
$\n{x}\in \gamma_j$, then
\begin{equation*}
q^{T}(\n{x})= \sum_{\n{\alpha}\in\cI^{j}_1} c^{\gamma_j}_{\n{\alpha}} B_{\n{\alpha}}^{\gamma_j,j}(\n{x})=q^{\gamma_j}(\n{x}),
\end{equation*}
and the result follows.
\end{proof}

The algorithm for~\texttt{BBExtension} corresponding to Lemma~\ref{Lemmaextension} is
presented in Algorithm~\ref{Extensionqj}.
\begin{algorithm}
    \SetKwInOut{Input}{Input}
    \SetKwInOut{Output}{Output}
    \SetKwInOut{Comment}{Comment}
    \Input{$c^{\gamma_j}, \n{z}_1,\n{z}_2, \kappa$.}
    \Output{BB-form of polynomial $q^T$, $c^{T}_{\n{\alpha}}$.}
    \For{$\n{\alpha}\in\cI_2^{j}$}
    {
        \If{$\alpha_{\kappa} = 0$}
        {
            \If{$\kappa=3$}
            {
                $c^{T}_{\n{\alpha}}\gets c^{\gamma_j}_{(\alpha_{1},\alpha_{2})} \lambda_1(\n{z}_1)^{-\alpha_{1}}  \lambda_{2}(\n{z}_2)^{-\alpha_{2}}$\;
            }
            \If{$\kappa=2$}
            {
                $c^{T}_{\n{\alpha}}\gets c^{\gamma_j}_{(\alpha_{3},\alpha_{1})} \lambda_3(\n{z}_1)^{-\alpha_{3}}  \lambda_1(\n{z}_1)^{-\alpha_{1}}$\;
            }
            \If{$\kappa=1$}
            {
                $c^{T}_{\n{\alpha}}\gets c^{\gamma_j}_{(\alpha_{2},\alpha_{3})} \lambda_2(\n{z}_1)^{-\alpha_{2}}  \lambda_3(\n{z}_1)^{-\alpha_{3}}$\;
            }
        }
    }
    return $c^{T}_{\n{\alpha}}$\;
\caption{\texttt{BBExtension}( $\mathcal{A}_j$)}\label{Extensionqj}
\end{algorithm}

\subsubsection{\texttt{GcapT}}
The practical implementation of Lemma~\ref{Lemmaextension} requires the
identification of $\n{z}_1$ and $\n{z}_2$ satisfying $\gamma_j\cap
T=conv(\n{z}_1,\n{z}_2)$. We observe that $\gamma_j$ separates the vertices of $T$ in one of the following:
$\{\n{v}_1,\,\n{v}_2\}\cup\{\n{v}_3\}$; $\{\n{v}_2,\,\n{v}_3\}\cup\{\n{v}_1\}$
and $\{\n{v}_3,\,\n{v}_1\}\cup\{\n{v}_2\}$. Suppose that the isolated node is
given by $\n{v}_k$, $k\in\{1,2,3\}$. Let $\n{z}$ denote either $\n{z}_1$ or
$\n{z}_2$, and let the barycentric coordinates of $\n{z}$ be $\lambda_1$,
$\lambda_2$ and $\lambda_3$, then
\begin{equation}\label{eq:sysk}
\left.
\begin{array}{rcl}
	\lambda_k &=& 0                     \\
	\sum_{i=1}^3 G_{e_i}\lambda_i& = &0 \\
	\sum_{i=1}^3 \lambda_i &= &1        \\
\end{array}
\right\} \quad\mbox{subject to}\quad 0\leq\lambda_i\leq 1, \quad \mbox{for } i=1,2,3
\end{equation}
where $\n{e}_k\in\cI_2^1$ is the multi-index $(\n{e}_k)_i=\delta_{k,i}$ for
$i=1,2,3$.  The first condition holds because $\n{z}$ is on the edge opposite
to vertex $\n{v}_k$, the second because $\Gamma_j(\n{z})=0$ and the third by
the property of barycentric coordinates. Since there exist only \emph{two} such
points $\n{z}$, we know that the system~\eqref{eq:sysk} is solvable for
precisely \emph{two} of the cases $k\in\{1,2,3\}$. These observations provide
the simple approach to the identification of $\n{z}_1$ and $\n{z}_2$
implemented in Algorithm \ref{Computez1z2}.
\begin{algorithm}
    \SetKwInOut{Input}{Input}
    \SetKwInOut{Output}{Output}
    \SetKwInOut{Comment}{Comment}
    \Input{Control points $G$ of $\Gamma$.}
    \Output{Computation of intersection points $\n{z}_1,\,\n{z}_2$ and index of isolated vertex $\kappa$.}
    \For{$k=1,2,3$}
    {
	\eIf{\eqref{eq:sysk} is solvable for $k$} {
            $\n{z}^{k}_{temp} = \sum_{i=1}^3 \n{v}_i\lambda_i$\;
        }{
            $\kappa = k$\;
        }
    }
    $\kappa_1 = \kappa+1 \mod{3}$\;
    $\kappa_2 = \kappa+2 \mod{3}$\;
    $\n{z}_1 \gets \n{z}_{temp}^{\kappa_1}$\;
    $\n{z}_2 \gets \n{z}_{temp}^{\kappa_2}$\;
    return $\n{z}_1,\,\n{z}_2,\,\kappa$\;
\caption{\texttt{GcapT}(G)}\label{Computez1z2}
\end{algorithm}

\subsubsection{\texttt{Transform1D}} In order to use the one dimensional Newton-Bernstein algorithm we need to transform the two dimensional nodes lying on $\gamma_j$ to one dimensional. Subroutine \texttt{Transform1D} transforms the nodes $\{\n{x}_j\}$ on the segment $[\n{z}_1,\n{z}_2]$ to the nodes $\{x_i\}$ in $[0,1]$. This is implemented in Algorithm \ref{Transform1D}.
\begin{algorithm}
    \SetKwInOut{Input}{Input}
    \SetKwInOut{Output}{Output}
    \SetKwInOut{Comment}{Comment}
    \Input{Two dimensional interpolation nodes and data $(\n{x}_i,f_i)_{i=1}^{j+1}=\mathcal{A}_j$.}
    \Output{One dimensional interpolation data $\bar{\mathcal{A}}_j$.}
    \For{$k=1,j+1$}
    {
        $x_k \gets \| \n{x}_k-\n{z}_1 \|_{2} /  \| \n{z}_2-\n{z}_1 \|_{2}$\;
    }
	$\bar{\mathcal{A}}_j\gets (x_i,f_i)_{i=1}^{j+1}$\;
    return $\bar{\mathcal{A}}_j$\;
\caption{\texttt{Transform1D}($\mathcal{A}_j$)}\label{Transform1D}
\end{algorithm}

\subsubsection{\texttt{BBProduct} - Bernstein form of the product}\label{sectionBB-prod}
It remains to define a procedure for computing the BB-form of the product  $q_j
\Gamma$ in terms of the BB-form of $q_j\in \mathbb{P}^{j}(T)$ and $\Gamma\in
\mathbb{P}^{1}(T)$. This can be accomplished by means of formula \eqref{dprod}.
Let $\{c_{\n{\alpha}}\}_{\n{\alpha}\in \cI^j_2}$
and$\{G_{\n{\alpha}}\}_{\n{\alpha}\in \cI^1_2}$ be the control points of $q_j$ and
$\Gamma$, respectively. Then the control points of the product are given by
\begin{align}
 a_{\n{\alpha}}&=  \sum_{k=1}^{3}c_{\n{\alpha}-\n{e}_k} G_{\n{e}_k}
	\frac{\alpha_k}{ j+1},\quad \n{\alpha}\in \cI^{j+1}_2
	\label{prodqjGamma}
\end{align}
where terms involving negative multi-indices are treated as being zero. This
expression forms the basis for Algorithm~\ref{BB-prod} which computes the desired
product for the general $d$-dimensional case.
\begin{algorithm}
    \SetKwInOut{Input}{Input}
    \SetKwInOut{Output}{Output}
    \SetKwInOut{Comment}{Comment}
    \Input{BB-form of $q\in\mathbb{P}^j(T)$, $\{c_{\n{\alpha}}\}_{\n{\alpha}\in \cI^j_d}$ and  \\ BB-form of $\Gamma\in \mathbb{P}^{1}(T)$, $\{G_{\n{\alpha}}\}_{\n{\alpha}\in \cI^1_d}$.}
    \Output{BB-form of the product $q_j \Gamma$, $\{a_{\n{\alpha}}\}_{\n{\alpha}\in \cI^{j+1}_d}$.}
    \For{$\n{\alpha}\in \cI^{j+1}_{d}$}
    {
        \For{$k\gets1,d+1$}
        {
            \If{$\alpha_k > 0$}
            {
                $a_{\n{\alpha}}\gets a_{\n{\alpha}} +c_{\n{\alpha}-\n{e}_k}G_{\n{e}_k}\frac{\alpha_k}{j+1}$\;
            }
        }
    }
    return $\{a_{\n{\alpha}}\}_{\n{\alpha}\in \cI^{j+1}_d}$\;
\caption{\texttt{BBProduct}($\{c_{\n{\alpha}}\}_{\n{\alpha}\in \cI^j_d},\,\{G_{\n{\alpha}}\}_{\n{\alpha}\in \cI^1_d}$)}\label{BB-prod}
\end{algorithm}


\subsection{Numerical example:  2D simplex}\label{Numericssimplex}
In this section we consider the two dimensional Bernstein-B\'ezier interpolation problem over a simplex with non-trivial interpolation data. We compare performance of Algorithm \ref{SS2D} with command ``$\backslash$'' in \texttt{Matlab}.
\begin{example}\label{ex6}\normalfont
We consider $n=10$ and a distribution of points satisfying the Solvability Condition (S). We plot the interpolation points in Figure \ref{ex6fig}. Similarly to Example \ref{ex1} we consider as a load vector $f_1$ and $f_2$ with the alternating sign property. We present and compare our results in Table \ref{tab:ex6}.

\begin{figure}
  \begin{minipage}[c]{0.55\textwidth}
    \includegraphics[width=\textwidth]{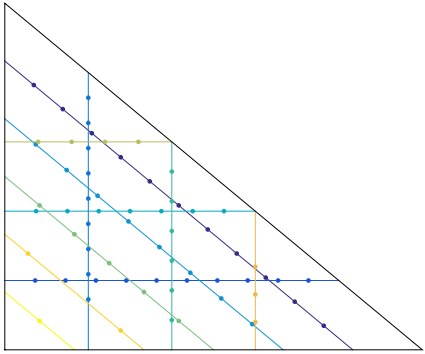}
  \end{minipage}\hfill
  \begin{minipage}[c]{0.45\textwidth}
    \caption{Example \ref{ex6}. Distribution of  interpolation points.} \label{ex6fig}
  \end{minipage}
\end{figure}

\begin{table}[htbp]
  \centering
  \begin{tabular}{c  @{\hskip .7in}  c @{\hskip .7in}   c}
    \toprule
    $f_i$&  $A\backslash f_i$  & \texttt{S2D}-\texttt{NewtonBernstein}\\
    \midrule
    $f_1$ &  4.9e-11    &  4.9e-13 \\
    $f_2$ &  3.1e-11    &  3.3e-13 \\
 \bottomrule
  \end{tabular}\vskip2mm
  \caption{Example \ref{ex6}: Relative errors in $L^2$ norm.}
  \label{tab:ex6}
\end{table}

We observe in Table \ref{tab:ex6} the superior accuracy of Algorithm \ref{SS2D} with respect to the \texttt{Matlab} solver. 
\end{example}

\subsection{Generalisation to $d$-dimensions}

The extension of the procedure described in Algorithm~\ref{SS2D} to the general
case is not difficult. Firstly, the Solvability Condition (S) is generalised to
$d$-dimensions recursively through requiring the existence of a non-overlapping
partition of the interpolation data into sets $\cA_n,\cA_{n-1},...,\cA_0$ of
nodes of appropriate dimension on a sub-simplex augmented by the solvability
condition on each of the $(d-1)$-subsimplices for the interpolation problems
associated with the sets $\cA_j$, $j=n,n-1,...,0$.  Formula
\eqref{Newtonform2D} can likewise be extended to higher dimensions with the
functions $\Gamma_j$ replaced by affine functions vanishing on the appropriate
sub-simplex. Thirdly, by utilising Algorithm \ref{SS2D} to solve the
$(d-1)$-dimensional sub-interpolation problems along with an obvious extension
of Lemma~\ref{Lemmaextension}, we obtain a subroutine equivalent to
\texttt{BBExtension} which extends the polynomial on the $(d-1)$-simplex to the
$d$-simplex.  Finally, a subroutine extending \texttt{BBAffine} to
$d$-dimension is easily obtained and subroutine \texttt{BBProduct} is written in
terms of $d$-dimensions.

\nocite{*}
\bibliography{Bibliography}{}
\bibliographystyle{plain}

\end{document}